\documentclass[12pt]{article}

\usepackage[margin=1in]{geometry}
\usepackage{natbib,amsthm,amsmath,amssymb,url}
\usepackage[colorlinks = {true}, linkcolor = {blue}, citecolor = {blue}]{hyperref}

\theoremstyle{plain}
\newtheorem{theorem}{Theorem}
\newtheorem{lemma}{Lemma}

\newtheorem{corollary}{Corollary}

\theoremstyle{remark}

\newcommand{\bbH}{\mathbb{H}}
\newcommand{\bbB}{\mathbb{B}}
\newcommand{\bbU}{\mathbb{U}}

\newcommand{\gpsv}{\textit{GPvs}(a_1, a_2, \pi_B)}
\newcommand{\gplp}{\textit{GPlp}(a_1, a_2, \pi_B, \pi_Q)}

\newcommand{\gpsvstar}{\textit{GPvs}^*(a_1, a_2, \pi_B)}
\newcommand{\gplpstar}{\textit{GPlp}^*(a_1, a_2, \pi_B, \pi_Q)}

\DeclareMathOperator*{\plim}{plim} 
\newcommand{\rline}{\mathbb{R}}
\newcommand{\expct}{\mathbb{E}}
\newcommand{\scO}{\mathcal{O}}
\newcommand{\scB}{\mathcal{B}}
\newcommand{\scQ}{\mathcal{Q}}

\newcommand{\beq}{\begin{equation}}
\newcommand{\eeq}{\end{equation}}
\newcommand{\diag}{\mathrm{diag}}
\newcommand{\gam}{\textit{Ga}}
\newcommand{\nm}{\textit{N}}
\newcommand{\bern}{\textit{Ber}}

\newcommand{\td}{{\tilde d}}
\newcommand{\tq}{{\tilde q}}

\title{Dimension adaptability of Gaussian process models with variable selection and projection}
\author{Surya T Tokdar\\ {\it \small Duke University}}
\date{}

\begin{document}
\maketitle

\begin{abstract}
It is now known that an extended Gaussian process model equipped with rescaling can adapt to different smoothness levels of a function valued parameter in many nonparametric Bayesian analyses, offering a posterior convergence rate that is optimal (up to logarithmic factors) for the smoothness class the true function belongs to. This optimal rate also depends on the dimension of the function's domain and one could potentially obtain a faster rate of convergence by casting the analysis in a lower dimensional subspace that does not amount to any loss of information about the true function. In general such a subspace is not known a priori but can be explored by equipping the model with variable selection or linear projection. We demonstrate that for nonparametric regression, classification, density estimation and density regression, a rescaled Gaussian process model equipped with variable selection or linear projection offers a posterior convergence rate that is optimal (up to logarithmic factors) for the lowest dimension in which the analysis could be cast without any loss of information about the true function. Theoretical exploration of such dimension reduction features appears novel for Bayesian nonparametric models with or without Gaussian processes.\\

\end{abstract}

\section{Introduction}
\label{intro}

%Bayesian nonparametric methodology is driven by construction of prior distributions on function spaces. Toward this, Gaussian process distributions have proved extremely useful due to their mathematical and computational tractability and ability to incorporate a wide range of smoothness assumptions.
Gaussian processes are widely used in Bayesian analyses for specifying prior distributions over function valued parameters. Examples include
 %Gaussian process models have been widely used in 
 {\it spatio-temporal modeling} \citep{handcock.stein.93, kim.etal.05, banerjee&etal08}, {\it computer emulation} \citep{sacks.etal89, kennedy&ohagan01, oakley&ohagan02, gramacy.lee08}, {\it nonparametric regression and classification} \citep{neal98, csato.etal.00, rasmussen&williams06, short.etal.07}, {\it density estimation} \citep{lenk88, tokdar07}, {\it density and quantile regression} \citep{tokdar&etal10, tokdar&kadane11}, {\it functional data analysis} \citep{shi&wang08, petrone&etal09} and {\it image analysis} \citep{sudderth&jordan09}. \citet{rasmussen&williams06} give a thorough overview of likelihood based exploration of Gaussian process models, including Bayesian treatments.

Theoretical properties of many Bayesian Gaussian process models have been well researched \citep[see][and the references therein]{tokdar&ghosh07, choi&schervish07, ghosal&roy06, vandervaart&vanzanten08, vandervaart&vanzanten09, dejonge.zanten10, castillo.11}. In particular, \cite{vandervaart&vanzanten09} present a remarkable adaptation property of  such models for nonparametric regression, classification and density estimation. They show a common Gaussian process (GP) prior specification equipped with a suitable rescaling parameter offers posterior convergence at near optimal minimax asymptotic rates across many classes of finitely and infinitely differentiable {\it true} functions. The rescaling parameter is a stochastic counterpart of a global bandwidth parameter commonly seen in smoothing-based non-Bayesian methodology. However, a single prior distribution on the rescaling parameter is enough to ensure near optimal convergence across all these classes of functions.

In this article we explore additional adaptation properties of GP models that are also equipped with variable selection or linear projection. To appreciate the practical utility of this exercise, consider a nonparametric (mean) regression model $Y_i = f(X_i) + \xi_i$, $i \ge 1$, where $X_i$'s are $d$-dimensional and $\xi_i$'s are independent draws from a zero mean normal distribution. When $f$ is assigned a suitable GP prior distribution equipped with a rescaling parameter and the true conditional mean function $f_0$ is H\"older $\alpha$ smooth (Section \ref{main set up}), the posterior distribution on $f$ converges to $f_0$ at a rate $n^{-\alpha/(2\alpha + d)} (\log n)^k$. This rate, without the $\log n$ term is optimal for such an $f_0$ in a minimax asymptotic sense \citep{stone82}. Now suppose $f_0(X_i)$ depends on $X_i$ only through its first two coordinates $Z_i$. If this information was known, we could cast the model as $Y_i = g(Z_i) + \xi_i$ and assign $g$ with a GP prior distribution with rescaling to obtain a faster convergence rate of $n^{-\alpha / (2\alpha + 2)}(\log n)^{k_1}$. If in addition we knew that $f_0(X_i)$ depends only on the difference $U_i$ of the first two coordinates of $X_i$, then we would instead cast the model as $Y_i = h(U_i) + \xi_i$ and with a rescaled GP prior on $h$ obtain an even faster convergence rate of $n^{-\alpha / (2\alpha + 1)}(\log n)^{k_2}$.

In practice, we do not know what sort of lower dimensional projections of $X_i$ perfectly explain the dependence of $f_0(X_i)$ on $X_i$. But this could be explored by extending the GP model to include selection of variables \citep{linkletter06} or linear projection onto lower dimensional subspaces \citep{tokdar&etal10}. The questions we seek to answer are as follows. Do GP models equipped with rescaling and variable selection offer a posterior convergence rate of $n^{-\alpha/(2\alpha + d_1)}(\log n)^{k_3}$ when the true $f$ is a H\"older $\alpha$-smooth function $f_0$ that depends only on $d_1 \le d$ many coordinates of its argument? More generally, do GP models equipped with rescaling and linear projection offer a posterior convergence rate of $n^{-\alpha / (2\alpha + d_0)}(\log n)^{k_4}$ when the true $f$ is a H\"older $\alpha$-smooth function such that $f_0(X_i)$ depends on a rank-$d_0$ linear projection of $X_i$? We demonstrate the answer to either question to be affirmative for extensions of the so called square exponential GP models in nonparametric mean regression, classification, density estimation and density regression.

Although projection or selection based dimension reduction is routinely employed in a variety of Bayesian nonparametric models with or without the use of Gaussian processes \citep[see for example,][]{rodriguez.dunson11}, their theoretical implications have not been fully explored. Best results so far demonstrate posterior consistency \citep{tokdar&etal10, pati.etal.11}, which already holds without these advanced features. Our results indicate that there is indeed an added advantage in terms of possibly faster posterior convergence rates. These results, with necessary details are presented in Section \ref{main}, which, we hope, can be appreciated by all readers interested in Bayesian nonparametric models with or without technical knowledge about Gaussian processes. Section \ref{basic} presents a set of deeper and more fundamental results, with non-trivial extensions of results presented in \cite{vandervaart&vanzanten09}. However, we have tried our best to make our calculations easily accessible to other researchers interested in studying extensions of GP models with additional adaptation features. We conclude in Section \ref{disc} with remarks on density regression versus density estimation and on a recent, unpublished work on a similar topic by \citet{anirban11}.

\section{Main results}
\label{main}
\subsection{Extending a rescaled GP with variable selection or projection}
\label{main set up}
We will restrict ourselves to nonparametric models where a function valued parameter $f$, to be modeled by a GP or its extensions, is defined over a compact subset of $\rline^d$ fom some $d$. Without loss of generality we can assume this set to be equal to $\bbU_d$, the unit disc $\{x \in \rline^d: \|x\| \le 1\}$ centered at the origin. If the actual domain of $f$ is not elliptic, such as a rectangle giving bounds on each coordinate of the argument $x$, we will simply shift and scale it to fit inside $\bbU_d$. Working on the larger domain $\bbU_d$ poses no technical difficulties.

Let $W = (W(t): t \in \rline^d)$ be a separable, zero mean Gaussian process with an isotropic, square exponential covariance function $\expct \{W(t)W(s)\} = \exp(-\|t - s\|^2)$. For any $a > 0$, $b \in \{0, 1\}^d$ and $q \in \scO_d$, define $W^{a,b,q} = (W^{a,b,q}(x): x \in \bbU_d)$ by 
\beq
W^{a,b,q}(x) = W(\diag(ab)\cdot qx),
\eeq 
where for any vector $v$, $\diag(v)$ denotes the diagonal matrix with the elements of $v$ on its diagonal. Note that $W^{a,b,q}(x) = W^{a,b,q}(z)$ if and only if $Rx = Rz$ where $R$ is the orthogonal projection matrix $q'\diag(b)q$. Therefore the law of $W^{a,b,q}$ defines a probability measure on functions $f: \bbU_d \to \rline$ such that $f(x)$ depends on $x$ only through the projection $R$. Also note that with $q = I_d$, the $d$-dimensional identity matrix, $R$ simply projects along the axes selected by $b$.

Let $|b|$ denote the number of ones in a $b \in \{0,1\}^d$. Suppose $(A,B,Q)$ are distributed as
\beq
(B, Q ) \sim \pi_B \times \pi_Q, \;\; A^{|b|}| (B = b, Q ) \sim \gam(a_1, a_2), 
\label{abq dist}
\eeq
independently of $W$, where $a_1 \ge 1, a_2 > 0$, $\pi_B$ is a strictly positive probability mass function on $\{0,1\}^d$ and $\pi_Q$ is a strictly positive probability density function on $\scO_d$. When $B = 0$, we simply take $A$ to be degenerate at 1. The law of the process $W^{A,B,Q}$, which extends the square exponential GP law by equipping it with rescaling and linear projection, will be denoted $\textit{GPlp}(a_1,a_2,\pi_B, \pi_Q)$. Similarly, the law of the process $W^{A, B, I_d}$, which extends the square exponential GP law by equipping it with rescaling and variable selection, will be denoted $\textit{GPvs}(a_1, a_2, \pi_B)$.

In the sequel, a function $f:U\to \rline$ defined on a compact subset $U$ of $\rline^d$ is called H\"older $\alpha$ smooth for some $\alpha > 0$ if it has bounded continuous derivatives (in the interior of $U$) up to order $\lfloor \alpha \rfloor$, the largest integer smaller than $\alpha$ with all its $\lfloor \alpha \rfloor$-th order partial derivatives being H\"older continuous with exponent no larger than $\alpha - \lfloor \alpha \rfloor$.

\subsection{Mean regression with Gaussian errors}

Nonparametric regression of a response variable on a vector of covariates with Gaussian errors comes in two flavors, depending on how the {\it design points}, i..e, the covariate values are obtained. They could either be fixed in an experimental study or measured as part of an observational study. The notion of posterior convergence differs slightly across the two contexts, a brief overview is given below.

\paragraph {\it Fixed design regression.} Suppose real-valued observations $Y_1, Y_2, \cdots$ are modeled as $Y_i = f(x_i) + \xi_i$ for a given sequence of points $x_1, x_2, \cdots$ from $\bbU_d$, with independent, $\nm(0, \sigma^2)$ errors $\xi_1,\xi_2,\cdots$. Assume $(f, \sigma)$ is assigned a prior distribution $\Pi_{f,\sigma}(df, d\sigma) = \Pi_f(df) \times \pi_{\sigma}(d\sigma)$ where $\Pi_f$ is either $\gpsv$ or $\gplp$ and $\pi_{\sigma}$ is a probability measure with a compact support inside $(0, \infty)$ and has a Lebesgue density that is strictly positive on this support.  

Let $\Pi^{n}_{f,\sigma}$ denote the posterior distribution of $(f, \sigma)$ given only the first $n$ observations $Y_1, \cdots, Y_n$, i.e.,
$$\Pi^n_{f,\sigma}(df, d\sigma) = \frac{\;\;\sigma^{-n} \exp\{-\frac1{2\sigma^2}\sum_{i = 1}^n (Y_i - f(x_i))^2\}\Pi_{f,\sigma}(df, d\sigma)}{\int \sigma^{-n} \exp\{-\frac1{2\sigma^2} \sum_{i = 1}^n (Y_i - f(x_i))^2\}\Pi_{f,\sigma}(df, d\sigma)}.$$
For every $n \ge 1$, define a design-dependent metric $\|\cdot\|_n$ on $\rline^{\bbU_d}$ as $\|f - g\|^2_n = \frac1n \sum_{i = 1}^n (f(x_i) - g(x_i))^2$. Let $(\epsilon_n: n \ge 1)$ be a sequence of positive numbers with $\lim_{n \to \infty} \epsilon_n = 0$ and $\lim_{n \to \infty} n\epsilon_n^2 = \infty$. For any fixed $f_0 : \bbU_d \to \rline$ and $\sigma_0 > 0$ we say the posterior converges at $(f_0, \sigma_0)$ at a rate $\epsilon_n$ (or faster) if for some $M > 0$,
$$ \plim_{n \to \infty} \Pi^n_{f,\sigma}(\{(f,\sigma): \|f - f_0\|_n + |\sigma - \sigma| \ge M \epsilon_n\}) = 0$$
whenever $Y_i = f_0(x_i) + \xi_i$ with independent $\xi_i \sim \nm(0, \sigma_0^2)$. Here ``$\plim$'' indicates {\it convergence in probability}.

\paragraph{\it Random design regression.} In the random design setting we have observations $(X_i, Y_i) \in U_d \times \rline$, $i = 1, 2, \cdots$ which are partially modeled as $Y_i = f(X_i) + \xi_i$ with $\nm(0, \sigma^2)$ errors $\xi_1, \xi_2, \cdots$. The design points $X_1, X_2, \cdots$ are assumed to be independent observations from an unknown probability distribution $G_x$ on $U_d$ and $X_i$'s and $\xi_i$'s are assumed independent. However inference on $G_x$ is not deemed important. Assume $(f, \sigma)$ is assigned a prior distribution $\Pi_{f,\sigma}$ as in the previous subsection and the corresponding posterior distribution based on the first $n$ observations $(X_1, Y_1), \cdots, (X_n, Y_n)$ is denoted $\Pi^n_{f,\sigma}$. 

Let $\|\cdot\|_{G_x}$ denote the $L_2$-metric with respect to $G_x$, i.e., $\|f - g\|_{G_x}^2 = \int (f(x) - g(x))^2 G_x(dx)$. Consider a sequence $(\epsilon_n > 0: n \ge 1)$ as before. For any fixed $f_0 : \bbU_d \to \rline$ and $\sigma_0 > 0$ we say the posterior converges at $(f_0, \sigma_0)$ at a rate $\epsilon_n$ (or faster) if for some $M > 0$,
$$ \plim_{n \to \infty} \Pi^n_{f,\sigma}(\{(f,\sigma): \|f - f_0\|_{G_x} + |\sigma - \sigma| \ge M \epsilon_n\}) = 0$$
whenever $Y_i = f_0(X_i) + \xi_i$ with $(\xi_i, X_i) \sim \nm(0, \sigma_0^2) \times G$ independently across $i \ge 1$.

Note that in either setting convergence at $(f_0, \sigma_0)$ at a rate $\epsilon_n$ also implies convergence of the (marginal) posterior distribution on $f$ to $f_0$ at the same rate (or faster). For either setting we can state the following dimension adaptation result.

\begin{theorem}
\label{thm reg}
Assume $f_0:\bbU_d \to \rline$ is a H\"older $\alpha$-smooth function on $U_d$ and $\sigma_0 >0$ is inside the support of $\pi_\sigma$. If $f_0(x)$ depends on $x$ only through $d_1 \le d$ many coordinates of $x$ and $\Pi_f = \gpsv$, the posterior converges at $(f_0, \sigma_0)$ at a rate $\epsilon_n = n^{-\alpha / (2\alpha + d_1)}(\log n)^{k}$ for every $k > d+1$. Furthermore, if $\alpha > 1$, $f_0(x)$ depends on $x$ only through a rank-$d_0$ linear projection $R x$  and $\Pi_f = \gplp$, the posterior converges at $(f_0, \sigma_0)$ at a rate $\epsilon_n = n^{-\alpha / (2\alpha + d_0)}(\log n)^{k}$ for every $k > d+1$.
\end{theorem}

%It is important to note that the metric $\|\cdot\|_{G_x}$ offers a useful measure of performance even if we do not know the distribution $G$. A typical goal of random design regression is to predict the value of $f$ at 
\subsection{Classification}
Suppose observations $(X_i, Y_i) \in \bbU_d \times \{0,1\}$, $i = 1, 2, \cdots$ are (partially) modeled as $Y_i \sim \bern(\Phi(f(X_i)))$, independently across $i$, where $\Phi$ is the standard normal or the logistic cumulative distribution function, with $X_i$'s assumed to be independent draws from a probability distribution $G_x$ on $\bbU_d$. Assume $f$ is assigned a prior distribution $\Pi_f$ which is either $\gpsv$ or $\gplp$, and let $\Pi^n_f$ denote the corresponding posterior distribution based on the first $n$ observations $(X_1, Y_1), \cdots, (X_n, Y_n)$, i.e.,
$$\Pi^n_f(df) = \frac{\;\;\left[\prod_{i = 1}^n  \Phi(f(X_i))^{Y_i}\{1 - \Phi(f(X_i))\}^{1 - Y_i}\right]\Pi_f(df)}{\int \left[\prod_{i = 1}^n  \Phi(f(X_i))^{Y_i}\{1 - \Phi(f(X_i))\}^{1 - Y_i}\right]\Pi_f(df)}.$$
Consider a sequence $(\epsilon_n > 0: n \ge 1)$ as before. For any fixed $f_0 : \bbU_d \to \rline$ and $\sigma_0 > 0$ we say the posterior converges at $f_0$ at a rate $\epsilon_n$ (or faster) if for some $M > 0$,
$$ \plim_{n \to \infty} \Pi^n_f(\{f: \|f - f_0\|_{G_x} \ge M \epsilon_n\}) = 0$$
whenever $Y_i | X_i \sim \bern(\Phi(f_0(X_i)))$ and $X_i \sim G$, independently across $i \ge 1$.

\begin{theorem}
\label{thm class}
Let $f_0:\bbU_d \to \rline$ be a H\"older $\alpha$-smooth function on $U_d$. If $f_0(x)$ depends on $x$ only through $d_1 \le d$ many coordinates of $x$ and $\Pi_f = \gpsv$, the posterior converges at $f_0$ at a rate $\epsilon_n = n^{-\alpha / (2\alpha + d_1)}(\log n)^k$ for every $k > d + 1$. Furthermore, if $\alpha > 1$, $f_0(x)$ depends on $x$ only through a rank-$d_0$ linear projection $Rx$  and $\Pi_f = \gplp$, the posterior converges at $f_0$ at a rate $\epsilon_n = n^{-\alpha / (2\alpha + d_1)}(\log n)^k$ for every $k > d + 1$.
\end{theorem}

\subsection{Density or point pattern intensity estimation}

Consider observations $X_i \in \bbU_d$, $i = 1, 2, \cdots$ modeled as independent draws from a probability density $g$ on $\bbU_d$ that can be written as
$$g(x) = \frac{g^*(x) \exp(f(x))}{\int_{\bbU_d} g^*(z)\exp(f(z))dz}$$
for some fixed, non-negative function $g^*$ and some unknown $f:\bbU_d \to \rline$. This type of models also arise in analyzing spatial point pattern data with non-homogeneous Poisson process models where the intensity function is expressed as $g^*(x)\exp(f(x))$. Assume $f$ is assigned the prior distribution $\Pi_f$ which is either $\gpsv$ or $\gplp$, and let $\Pi_g$ denote the induced prior distribution on $g$. The corresponding posterior distribution based on $X_1, \cdots, X_n$ is given by
$$\Pi^n_g(dg) = \frac{\prod_{i = 1}^n g(X_i)\Pi_g(dg)}{\int \prod_{i = 1}^n g(X_i)\Pi_g(dg)}.$$
Let $h(g_1, g_2) = \{\int_{\bbU_d} (g^{1/2}_1(x) - g^{1/2}_2(x))^2 dx\}^{1/2}$ denote the Hellinger metric. 

Consider a sequence $(\epsilon_n > 0: n \ge 1)$ as before. For any fixed density $g_0$ on $\bbU_d$, we say the posterior converges at $g_0$ at a rate $\epsilon_n$ (or faster) if for some $M > 0$,
$$\plim_{n \to \infty} \Pi^n_g(\{g : h(g, g_0) \ge M\epsilon_n\}) = 0$$
whenever $X_i$'s are independent draws from $g_0$.

\begin{theorem}
\label{thm denest}
Let $g_0$ be a probability density on $\bbU_d$ satisfying $g_0(x) = g^*(x)e^{f_0(x)} / \int g^*(z)e^{f_0(z)}dz$ for some H\"older $\alpha$-smooth $f_0 :\bbU_d \to \rline$. If $f_0(x)$ depends on $x$ only through $d_1 \le d$ many coordinates of $x$ and $\Pi_f = \gpsv$, the posterior converges at $g_0$ at a rate $\epsilon_n = n^{-\alpha / (2\alpha + d_1)}(\log n)^k$ for every $k > d + 1$. Furthermore, if $\alpha > 1$, $f_0(x)$ depends on $x$ only through a rank-$d_0$ linear projection $R x$  and $\Pi_f = \gplp$, the posterior converges at $g_0$ at a rate $\epsilon_n = n^{-\alpha /(2\alpha + d_0)}(\log n)^k$ for every $k > d + 1$.
\end{theorem}

In the above theorem, the conditions on $f_0$ are equivalent to saying that $g_0(x) / g^*(x)$ varies in $x$ only along a linear subspace of the variable $x$. In the context of two dimensional point patter models, this implies that the intensity function, relative to $g^*$, is constant over the spatial domain or constant along a certain direction.

\subsection{Density regression}

Consider again observations $(X_i, Y_i) \in \bbU_d \times \rline$, $i = 1, 2, \cdots$ where we want to develop a regression model between $Y_i$'s and $X_i$'s.  In density regression the entire conditional density, and not just the conditional mean of $Y_i$ given $X_i$ is modeled nonparametrically. \cite{tokdar&etal10} consider the model $Y_i | X_i \sim g(\cdot | X_i)$, independently across  $i$, where the conditional densities $g(\cdot | x)$, $x \in \bbU_d$ are given by point by point logistic transforms of a function $f : \bbU_d \times [0,1] \to \rline$:
\beq
g(y | x) = \frac{g^* (y)\exp\{f(x, G^* (y))\}}{\int_{-\infty}^\infty g^* (z) \exp\{f(x,  G^* (z))\}dz},\;\;y \in \rline,
\label{den reg def}
\eeq
for some fixed probability density $g^* $ on $\rline$ with cumulative distribution function $G^* $. To construct a suitable prior distribution for $f$, we consider an extension of the process $W^{a, b, q}$. 

Let $Z = (Z(t, u): t \in \rline^d, u \in [0,\infty))$ be a separable, zero mean Gaussian process with isotropic, square-exponential covariance function $\expct \{Z(t, u)Z(s, v)\} = \exp( - \|t - s\|^2 - \|u - v\|^2)$. Define $Z^{a,b,q} = (Z^{a,b,q}(x, u): x \in \bbU_d, u \in [0,1])$ as
\beq Z^{a,b,q}(x, u) = Z(\diag(ab) \cdot qx, au).\eeq
Let $\gpsvstar$ and $\gplpstar$, respectively, denote the laws of the processes $Z^{A, B, Q}$ and $Z^{A, B, I_d}$ where $(B, Q)$ are distributed as in (\ref{abq dist}) and $A^{|b| + 1} | (B = b, Q) \sim \gam(a_1,a_2)$.

Now suppose $f$ in (\ref{den reg def}) is assigned the prior distribution $\Pi_f$ which is either $\gpsvstar$ or $\gplpstar$, and denote the induced prior distribution on $g = (g(y|x): x \in \bbU_d, y \in \rline)$ by $\Pi_g$. The corresponding posterior distribution $\Pi^n_g$ based on $(X_1, Y_1), \cdots, (X_n, Y_n)$ is given by
$$\Pi_g^n(dg) = \frac{\{\prod_{i = 1}^n g(Y_i | X_i)\}\Pi_g(dg)}{\int \{\prod_{i = 1}^n g(Y_i | X_i)\}\Pi_g(dg)}.$$
Let $\rho_{G_x}(\cdot, \cdot)$ denote the metric $\rho^2_{G_x}(g_1, g_2) = \int \{g^{1/2}_1(y|x) - g^{1/2}_2(y|x)\}^2 G_x(dx)$ for a probability distribution $G_x$ on $\bbU_d$. 

Consider a sequence $(\epsilon_n > 0: n \ge 1)$ as before. For any fixed $g_0 = (g_0(y|x): x \in \bbU_d, y \in \rline)$, we say the posterior converges at $g_0$ at a rate $\epsilon_n$ (or faster) if  for some $M > 0$, 
$$\plim_{n \to \infty} \Pi^n_g(\{g : \rho_{G_x}(g, g_0) \ge M \epsilon_n\}) = 0$$ whenever $Y_i | X_i \sim g_0(\cdot|X_i)$ and $X_i \sim G$, independently across $i \ge 1$.

\begin{theorem}
\label{thm denreg}
Let $g_0 = (g_0(y|x): x \in \bbU_d, y \in \rline)$ satisfy 
$$g_0(y | x) = \frac{g^*(y) \exp\{f_0(x, G^*(y))\} }{ \int g^*(z) \exp\{f_0(x, G^*(z))\} dz}$$
for an $f_0:\bbU_d\times[0, 1] \to \rline$ that is H\"older $\alpha$-smooth. If $f_0(x, u)$ depends on $x$ only through $d_1 \le d$ many coordinates of $x$ and $\Pi_f = \gpsvstar$, the posterior converges at $g_0$ at a rate $\epsilon_n = n^{-\alpha / (2\alpha + d_1 + 1)}(\log n)^k$ for every $k > d + 2$. Furthermore, if $\alpha > 1$, $f_0(x, u)$ depends on $x$ only through a rank-$d_0$ linear projection $R x$  and $\Pi_f = \gplpstar$, the posterior converges at $g_0$ at a rate $\epsilon_n = n^{-\alpha/(2\alpha + d_0 + 1)} (\log n)^k$ for every $k > d + 2$.  
\end{theorem}

\section{Adaptation properties of GP extensions}
\label{basic}

\cite{ghosal&etal00}, later refined by \cite{ghosal.vandervaart07b}, provide a set of three sufficient conditions that can be used to establish posterior convergence rates for Bayesian non-parametric models for independent observations. One of these conditions relates to prior concentration at the true function, and the other two relate to existence of a sequence of compact sets which have relatively small sizes but receive large probabilities from the prior distribution. 

For the results stated in Theorems \ref{thm reg}, \ref{thm class} and \ref{thm denest} relating respectively to mean regression, classification or density estimation, these three sufficient conditions map one to one \citep{vandervaart&vanzanten08} to the following conditions on an extended GP $\tilde W$ with law $\gpsv$ or $\gplp$ as appropriate, the true function $f_0$ and the desired rate $\epsilon_n$: there exist sets $\scB_n \subset \rline^{\bbU_d}$ and a sequence $(\tilde\epsilon > 0: n \ge 1)$ with $\tilde\epsilon_n < \epsilon_n$, $\lim_{n \to \infty} n\tilde \epsilon_n^2 = \infty$ such that for all sufficiently large $n$,
\begin{align}
P(\|\tilde W - f_0\|_\infty \le \tilde\epsilon_n) & \ge e^{-n\tilde\epsilon_n^2},\label{c1}\\
P(\tilde W \not\in \scB_n) & \le e^{-4n\tilde\epsilon_n^2},\label{c2}\\
\log N(\epsilon_n, \scB_n, \|\cdot\|_\infty) & \le n \epsilon_n^2,\label{c3}
\end{align}
where $N(\epsilon, B, \rho)$ denotes the minimum number of balls of radius $\epsilon$ (with respect to a metric $\rho$) needed to cover a set $B$. For the density regression results stated in Theorem \ref{thm denreg}, the sufficient conditions also map one to one to the above but with $\tilde W$ now following either $\gpsvstar$ or $\gplpstar$ and with $\scB_n \subset \rline^{\bbU_d \times [0,1]}$. This can be proved along the lines of Theorem 3.2 of \citet{vandervaart&vanzanten08}, by looking at the joint density of $(X_i, Y_i)$ determined by $G_x$ and $g \sim \Pi_g$.

We verify these conditions in the following subsections by extending the calculations presented in \cite{vandervaart&vanzanten09}. A fundamental ingredient of these calculations is the reproducing kernel Hilbert space (RKHS) associated with a Gaussian process. The RKHS of $W$ is defined as the set of functions $h : \rline^d \to \rline$ that can be represented as $h(t) = \expct \{W(t) L\}$ for some $L$ in the closure of $\{V = a_1 W(t_1) + \cdots + a_k W(t_k) : k \ge 1, a_i \in \rline, t_i \in \rline^d\}$. Similar definitions apply to the processes $W^{a,b,q}$ and $Z^{a,b,q}$ with domains $\bbU_d$ and $\bbU_d \times [0,1]$ instead of $\rline^d$.

In Lemma 4.1 of \cite{vandervaart&vanzanten09}, the RKHS of $W$ is identified as the set of functions $h$ such that $h(t) = Re\{\int e^{i(\lambda, t)}\psi(\lambda)d\mu(\lambda)\}$ for some $\psi \in L_2(\mu)$, where $Re\{z\}$ denotes the real part of a complex number $z$, $i$ is the square root of $-1$ and $\mu$ is the (unique) spectral measure on $\rline^d$ of $W$, satisfying $\expct\{W(t)W(s)\} = \int e^{-i(t - s, \lambda)}d\mu(\lambda)$. For the isotropic, square exponential GP $W$, the spectral measure is  the $d$-dimensional Gaussian probability measure with mean zero and variance matrix $2I_d$. The RKHS norm of such an $h$ is precisely $\|\psi\|_{L_2(\mu)}$. 

By simple change of variables it follows that the RKHS of $W^{a,b,q}$, for any $a > 0, b \in \{0,1\}^d$ and $q \in \scO_d$, is given by functions $h$ such that $h(x) = Re\{\int e^{i(\lambda, x)} \psi(\lambda)d\mu_{a,b,q}(\lambda)\}$ with RKHS norm $\|\psi\|_{L_2(\mu_{a,b,q})}$ where $\mu_{a,b,q}$ is the $d$-dimensional Gaussian probability measure with mean 0 and variance matrix $2a^2 q'\diag(b)q$. In the rest of the paper this RKHS is denoted $\bbH^{a,b,q}$ and $\bbH^{a,b,q}_1$ is used to denote its unit ball at the origin. Also, $\bbB$ is used to denote the Banach space of continuous functions on $\bbU_d$ equipped with the supremum norm $\|\cdot\|_\infty$. The unit ball at origin of this space is denoted $\bbB_1$.

\subsection{Variable selection extension}
\label{cov selec}

To start with let $W^{a,b}$ denote $W^{a,b,q}$ with $q$ fixed at the $d$-dimensional identity matrix and let $\bbH^{a,b}$ stand for the corresponding RKHS $\bbH^{a,b,q}$. For any $b \in \{0,1\}^d$ and $x \in \bbU_d $ let $x_b$ denote the $|b|$-dimensional vector of coordinates of $x$ selected by $b$. Also for any $\alpha > 0$ and any $\tilde d \in \{0, \cdots, d\}$ let $H_{\alpha, \td}$ denote the class of H\"older $\alpha$-smooth, real functions on $\bbU_{\td}$. Notice that if $f_0(x)$ depends only on $d_1$ many coordinates of $x$ then there exist $b_0 \in \{0,1\}^d$ with $|b_0| = d_1$ and $v_0 \in H_{\alpha, d_1}$ such that $f_0(x) = v_0(x_{b_0})$ for all $x \in \bbU_d$. 

\begin{theorem}
\label{vs}
Let $f_0 : \bbU_d  \to \rline$ satisfy $f_0(x) = v_0(x_{b_0})$ for some $b_0 \in \{0,1\}^d$ with $|b_0| = d_1$ and some $v_0 \in H_{\alpha, d_1}$. Then for every $s > 0$, there exist measurable subsets $\scB_n \subset \rline^{\bbU_d}$ and a constant $K > 0$ such that (\ref{c1})-(\ref{c3}) hold with $\tilde W = W^{A,B}$, $\tilde \epsilon_n = n^{-\alpha / (2\alpha + d_1)}(\log n)^{(d + 1)/2 + s}$ and $\epsilon_n = K \tilde \epsilon_n (\log n)^{(d + 1)/2}$.
\end{theorem}

\begin{proof}
Define $W^a_{b_0} = (W^a_{b_0}(u): u \in \bbU_{d_0} )$ by $W^a_{b_0}(u) = W^{a, b_0}(u^{b_0})$ where $u^{b_0}$ denotes the unique zero-insertion expansion of $u$ to a $d$-dimensional vector such that $(u^{b_0})_{b_0} = u$. For any $u \in \bbU_{d_0} $, for every $x \in \bbU_d $ with $x_{b_0} = u$ we have $W^{a,b_0}(x) = W^a_{b_0}(u)$ and $f_0(x) = v_0(u)$. So
$$P(\|W^{A,B } - f_0\|_\infty \le \tilde\epsilon_n)  \ge \pi_B(b_0)P(\|W_{b_0}^A - v_0\|_\infty \le \tilde\epsilon_n).$$
From calculations presented in Section 5.1 of \cite{vandervaart&vanzanten09} it follows $P(\|W_{b_0}^A - v_0\|_\infty \le \delta_n) \ge \exp(-n\delta_n^2)$ for $\delta_n$ a large multiple of $n^{-\alpha/(2\alpha + |b_0|)}(\log n)^{(1 + |b_0|)/(2 + |b_0| / \alpha)}$ and all sufficiently large $n$. This leads to (\ref{c1}) because $\tilde\epsilon_n$ is larger than any such $\delta_n$ by a power of $\log n$.

It follows from the proof of Theorem 3.1 of \cite{vandervaart&vanzanten09} that for some $C_0 > 0$, $a_0 > 1$, $0 < \epsilon_0 < 1/2$ and for every $b \in \{0,1\}^d \setminus\{0\}$, $r > a_0$, $0 < \epsilon < \epsilon_0$, $M^2 > C_0 r^{|b|} (\log (r/\epsilon))^{1 + |b|}$ and $\delta = \epsilon / (2|b|^{3/2}M)$ the set
$$\scB^b_{M,r,\epsilon,\delta} = \left\{(r/\delta)^{|b|/2}M\bbH^{r,b}_1 + \epsilon\bbB_1\right\}\cup(\cup_{a < \delta} M\bbH^{a,b}_1 + \epsilon\bbB_1)$$
satisfies
\begin{align}
P(W^{A, b} \not\in \scB ^b_{M, r, \epsilon,\delta}) & \le C_1 r^{(a_1 - 1)|b| + 1}e^{-C_2r^{|b|}} + e^{-M^2/8}\label{eq1}\\
\log N(3\epsilon, \scB ^b_{M, r, \epsilon, \delta}, \|\cdot\|_\infty) & \le C_3r^{|b|}\left( \log \frac{M^{3/2}\sqrt{2|b|^{3/2}r}}{\epsilon^{3/2}}\right)^{1 + |b|} + 2 \log \frac M{\epsilon},\label{eq2}
\end{align}
for universal constants $C_1, C_2, C_3$. For any $b \in \{0, 1\}^d \setminus \{0\}$ define $\scB^b_n = \scB^b_{M_n, r_n, \epsilon_n, \delta_n}$ where $r_n^{|b|}$ is a large multiple of $n\tilde\epsilon_n^2$, $M^2_n$ is a large multiple of $n\tilde\epsilon_n^2 (\log n)^{1 + |b|}$ and $\delta_n = \tilde\epsilon_n/ (2|b|^{3/2}M_n)$. Then by the above inequalities, $P(W^{A, b} \not\in \scB^b_n) \le \exp(-C_4n\tilde\epsilon_n^2)$ and $\log N(K\tilde\epsilon_n, \scB ^b_n, \|\cdot\|_\infty) \le C_5 n \tilde\epsilon_n^2 (\log n)^{1 + |b|}$ for some large constants $C_4$, $C_5$ and $K$. It is easy to construct a $\scB^b_n$ with similar properties when $b = 0$. Therefore (\ref{c2}), (\ref{c3}) hold with $\scB_n = \cup_{b \in \{0,1\}^d} \scB^b_n$.
\end{proof}

\begin{corollary}
\label{vs1}
Let $f_0 : \bbU_d\times[0,1]  \to \rline$ satisfy $f_0(x,u) = v_0(x_{b_0},u)$ for some $b_0 \in \{0,1\}^d$ with $|b_0| = d_1$ and some H\"older $\alpha$-smooth $v_0 : \bbU_{d_1} \times [0,1] \to \rline$. Then for every $s > 0$, there exist measurable subsets $\scB_n \subset \rline^{\bbU_d\times[0,1]}$ and a constant $K > 0$ such that (\ref{c1})-(\ref{c3}) hold with $\tilde W = Z^{A,B,I_d}$, $\tilde \epsilon_n = n^{-\alpha / (2\alpha + d_1 + 1)}(\log n)^{(d + 2)/2 + s}$ and $\epsilon_n = K \tilde \epsilon_n (\log n)^{(d + 2)/2}$.
\end{corollary}

\begin{proof}
A proof can be constructed exactly along the lines of the proof above. The extra variable does not alter calculations, except for increasing all dimensions by one, because the variable selection parameter does not operate on it.
\end{proof}

\subsection{Linear projection extension}

Our proof of Theorem \ref{vs} is made relatively straightforward by the fact that $B$ lives on a discrete set. This is no longer the case when we work with $W^{A,B,Q}$ with $Q$ taking values on a continuum. However the support of $Q$, namely $\scO_d$ is a well behaved compact set, a fact that we make good use of. To start with, here is a result that shows how to relate the RKHS $\bbH^{a,b,q}$ with the RKHS $\bbH^{a,b,\tq}$ when $q, \tq \in \scO_d$ are close to each other.

\begin{lemma}
For any $a > 0$, $b \in \{0,1\}^d$ and $q, \tq \in \scO_d$, $\bbH_1^{a,b,q} \subset \bbH_1^{a,b,\tilde q } + a\sqrt d \|q - \tilde q \|_S\bbB_1$ where $\|\cdot\|_S$ denotes the spectral norm on $\scO_d$.
\label{q lem}
\end{lemma}

\begin{proof}
Any $h \in \bbH_1^{a,b,q}$ can be expressed as $h(x) = Re\{\int e^{i(\lambda, x)} \psi(\lambda)d\mu_{a,b,q}(\lambda)\}$ for some $\psi \in L_2(\mu_{a,b,q})$ with norm no larger than 1. For any $x, z \in \bbU_d $, 
\[
|h(x) - h(z)| \le \int |(\lambda, x - z)||\psi(\lambda)|d\mu_{a,b,q}(\lambda) \le \|x - z\| \sqrt{\smallint \|\lambda\|^2 d\mu_{a,b,q}(\lambda)} \le a\sqrt d\|x - z\|
\]
where the second inequality follows from two applications of the Cauchy-Schwartz inequality. 

Define $\tilde h(x) = h(q'\tilde q x)$ and $\tilde\psi(\lambda) = \psi(q'\tilde q \lambda)$. Then $\tilde\psi \in L_2(\mu_{a,b,\tq})$ with norm no larger than 1 and $\tilde h(x)$ is the real part of
$$\int e^{i(\lambda, q'\tilde q x)} \psi(\lambda)d\mu_{a,b,q}(\lambda)  = \int e^{i(\tilde q 'q\lambda, x)} \tilde\psi(\tilde q' q\lambda)d\mu_{a,b,q}(\lambda) = \int e^{i(\lambda, x)} \tilde\psi(\lambda)d\mu_{a,b,\tq}(\lambda),$$
and therefore $\tilde h \in \bbH_1^{a,b,\tq}$. From this the result follows because for any $x \in \bbU_d $, $|h(x) - \tilde h(x)| \le a\sqrt d \|x - q'\tq x\| \le a \sqrt d \|q - \tq\|_S$.
\end{proof}

Next we present the counterpart of Theorem \ref{vs} for the linear projection extension. Notice that if $f_0(x)$ depends on $x$ only through a rank-$d_0$ linear projection $Rx$ then there exist $b_0 \in \{0, 1\}^d$ with $|b_0| = d_0$, $q_0 \in \scO_d$ and $v_0 \in H_{\alpha, d_0}$ such that $f_0(x) = v_0((q_0x)_{b_0})$.

\begin{theorem}
Let $f_0 : \bbU_d  \to \rline$ satisfy $f_0(x) = v_0((q_0x)_{b_0})$ for some $q_0 \in \scO_d$, $b_0 \in \{0,1\}^d$ with $|b_0| = d_0$ and some $v_0 \in H_{\alpha, d_0}$ with $\alpha > 1$. Then for every $s > 0$, there exist measurable subsets $\scB_n \subset \rline^{\bbU_d}$ and a constant $K > 0$ such that (\ref{c1})-(\ref{c3}) hold with $\tilde W = W^{A,B,Q}$, $\tilde\epsilon_n = n^{-\alpha / (2\alpha + d_0)}(\log n)^{(d + 1)/2 + s}$ and $\bar \epsilon_n = K\tilde\epsilon_n (\log n)^{(d + 1)/2}$.
\end{theorem}

\begin{proof}
Clearly,
$$P(\|W^{A,B,Q } - f_0\|_\infty \le \epsilon_n)  \ge \pi_B(b_0) \int_{\scO_d} P(\|W^{A, b_0, q} - f_0\|_\infty \le \epsilon_n)\pi_Q(q)dq.$$
For any $q \in \scO_d$, define $W^a_{b_0,q} = (W^a_{b_0,q}(u): u \in \bbU_{d_0} )$ and $f_q: \bbU_d \to \rline$ by $W^a_{b_0,q}(u) = W^{a, b_0, q}(q'u^{b_0})$ and $f_q(x) = v_0((qx)_{b_0})$. For any $u \in \bbU_{d_0} $ and every $x \in \bbU_d $ with $(qx)_{b_0} = u$ we have $W^{a,b_0,q}(x) = W^a_{b_0,q}(u)$ and $f_q(x) = v_0(u)$. Now, if $q \in \scO_d$ is such that $\|f_0 - f_q\|_\infty < \tilde\epsilon_n/2$ then
\begin{align*}
P(\|W^{A,b_0,q} - f_0\|_\infty \le \tilde\epsilon_n) \ge P(\|W^{A,b_0,q} - f_q\|_\infty \le \tilde\epsilon_n/2) = P(\|W^a_{b_0,q} - v_0\|_\infty \le \tilde\epsilon_n/2).
\end{align*}
This last probability does not depend on $q$ because $W$ is rotationally invariant and hence equals $P(\|W^a_{b_0,q_0} - v_0\|_\infty \le \tilde\epsilon_n/2) \ge e^{-n\delta_n^2}$ with $\delta_n$ a multiple of $n^{-\alpha/(2\alpha + |b_0|)}(\log n)^{(1 + |b_0|)/(2 + |b_0|/\alpha)}$, as in the previous theorem. From this (\ref{c1}) would follow if we can show $P(Q \in \{q: \|f_0 - f_q\|_\infty \le \tilde \epsilon_n/2\}) \ge e^{-n\delta_n^2}$. Note that $f_q(x) = f_0(q_0'qx)$ for all $x \in \bbU_d $. By assumption on $v_0$, $f_0$ has a bounded continuous derivative and hence $\|f_0 - f_q\|_\infty \le D_2 \|q_0 - q\|_S$ for some $D_2 < \infty$. But $P(\|Q - q_0\|_S \le \tilde\epsilon_n / (2D_2)) \ge D_3 \tilde\epsilon_n^{d(d - 1)/2}$ for some constant $D_3$ because a spectral ball of radius $\delta$ in $\scO_d$ has volume of the order $\delta^{d(d - 1)/2}$ for all small $\delta > 0$ and $\pi_Q$ is strictly positive on $\scO_d$. This completes the proof of the first assertion because $\tilde\epsilon_n^{d(d - 1)/2} \ge e^{-n\delta_n^2}$ as $-\log\tilde \epsilon_n / (n\delta_n^2) \to 0$.

To construct the sets $\scB_n$ we adapt the approach taken in the previous theorem to include $q$. In particular, for each $b \in \{0,1\}\setminus\{0\}$ and each $q \in \scO_d$, define
$$\scB^{b, q}_{M,r,\epsilon,\delta} = \left\{(r/\delta)^{|b|/2}M\bbH^{r,b,q}_1 + \epsilon\bbB_1\right\}\cup(\cup_{a < \delta} M\bbH^{a,b,q}_1 + \epsilon\bbB_1).$$
Inequalities (\ref{eq1}) and (\ref{eq2}) continue to hold with $\scB^b_{M,r,\epsilon,\delta}$ and $W^{A,b}$ respectively replaced with $\scB^{b,q}_{M,e,\epsilon,\delta}$ and $W^{A,b,q}$. Therefore by defining $\scB^{b,q}_n = \scB^{b,q}_{M_n,r_n,\epsilon_n,\delta_n}$ with $M_n, r_n$ and $\delta_n$ exactly as before we have $P(W^{A,b,q} \not\in \scB_n^{b,q}) \le \exp(-C_4n\epsilon_n^2)$ and $\log N(K\epsilon_n, \scB_n^{b,q}, \|\cdot\|_\infty) \le C_5 n\epsilon_n^2 (\log n)^{1 + |b|}$ for some finite constants $C_4, C_5$ and $K$. 

Fix any $b \in \{0,1\}^d \setminus \{0\}$ and take $\scB^b_n = \cup_{q \in \scO_d} \scB^{b,q}_n$. Then $P(W^{A, b, Q} \not\in \scB^b_n) \le \exp(-C_4n\epsilon_n^2)$. To bound the entropy of $\scB^b_n$, first get an $\zeta_n$-spectral norm net $\scQ_n$ of $\scO_d$ where $\zeta_n = \epsilon_n / \{M_nr_n\sqrt d (r_n/\delta_n)^{|b|/2}\}$. The size of $\scQ_n$ is no larger than $C_6 \zeta_n^{-d(d - 1)/2}$ for some universal constant $C_6$. For any $q \in \scO_d$ find $\tilde q  \in \scQ_n$ such that $\|q - \tilde q \|_S \le \zeta$. Then by Lemma \ref{q lem}, $\scB^{b, q}_n \subset \scB^{b,\tq}_n + \epsilon_n \bbB_1$ and hence,
$$\log N((K + 1)\epsilon_n, \scB^b_n, \|\cdot\|_\infty) \le \log \max_{\tq \in \scQ_n} N(K\epsilon_n, \scB^{b, \tq}_n, \|\cdot\|_\infty) + \log |\scQ_n|$$
which is smaller than $C_7 n\epsilon_n^2 (\log n)^{1 + |b|}$ for some constant $C_7$. Therefore (\ref{c2}), (\ref{c3}) hold with $\scB_n = \cup_{b \in \{0,1\}^d}\scB^b_n$.

\end{proof}

\begin{corollary}
\label{lp1}
Let $f_0 : \bbU_d\times[0,1]  \to \rline$ satisfy $f_0(x,u) = v_0((q_0x)_{b_0},u)$ for some $q_0 \in \scO_d$, $b_0 \in \{0,1\}^d$ with $|b_0| = d_1$ and some H\"older $\alpha$-smooth $v_0 : \bbU_{d_0} \times [0,1] \to \rline$. Then for every $s > 0$, there exist measurable subsets $\scB_n \subset \rline^{\bbU_d\times[0,1]}$ and a constant $K > 0$ such that (\ref{c1})-(\ref{c3}) hold with $\tilde W = Z^{A,B,Q}$, $\tilde \epsilon_n = n^{-\alpha / (2\alpha + d_0 + 1)}(\log n)^{(d + 2)/2 + s}$ and $\epsilon_n = K \tilde \epsilon_n (\log n)^{(d + 2)/2}$.
\end{corollary}

\begin{proof}
Again, the additional variable is unaffected by the projection parameter which operates only on the $x$ variable. So the above proof can be extended almost verbatim to prove this result.
\end{proof}

\section{Discussion}
\label{disc}

Besides variable selection and linear projection, another common way of extending Gaussian processes is to equip them with a vector of rescaling parameters, each operating along a single axis \citep{willras96}. In our notations, this could be defined as $W^a = (W^{a}(x): x \in \bbU_d)$ with $W^a(x)=W(\diag(a) \cdot x)$, for $a \in [0,\infty)^d$. The law of $W^A$, with $A$ assigned some prior distribution $\pi_A$, can be used as a prior distribution for a function valued parameter $f : \bbU_d \to \rline$. In an independent work, done in parallel to ours and posted on \url{arXiv.org}, \cite{anirban11} explore and establish some very interesting adaptability properties of such extensions. From a practitioner's point of view, the most interesting extension along this line would be $W^{A,Q}$ with $(A, Q) \sim \pi_A \times \pi_Q$ where for any $a \in [0,\infty)$, $q \in \scO_d$, one defines $W^{a,q}(x) = W(\diag(a) \cdot qx)$. One should be able to study theoretical properties of such an extension by combining the results presented in this paper and in \cite{anirban11}, but the details remain to be verified.

Note that we restrict to functions $f_0$ that are only finitely differentiable. For $f_0$ that are infinitely differentiable and satisfy some regularity conditions, a rescaled GP model offers a nearly parametric posterior convergence rate of $n^{-1/2}(\log n)^k$ for some $k$ that depends on $d$ and $f_0$ \citep{vandervaart&vanzanten09}. The dimension does not affect the leading term $n^{-1/2}$ and our techniques in Section \ref{basic} do not offer improvement in the logarithmic factor. 

It might seem a little underwhelming that for density regression our choice of metric $\rho_{G_x}(g_1, g_2)$ essentially defines the Hellinger metric between the joint densities $h_1(x, y) = g_1(y|x)$, and $h_2(x, y) = g_2(y|x)$, with respect to the product of $G_x$ and the Lebesgue measure on $\rline$, thus transporting the problem to one where one studies the joint density of $(X, Y)$. We make two observations to point out why this is not a terrible thing to do. First, no modeling is done on the unknown distribution $G_x$, the joint density view is purely a technical tool needed to map conditions required to prove Theorem \ref{thm denreg} to conditions (\ref{c1})-(\ref{c3}). Second, the goals of regression are well preserved despite the use of $G_x$ in defining $\rho_{G_x}$. Suppose one is interested in inference on $g(y|x^*)$ for test data $x^*$ generated from $G^*_x$, possibly different from $G_x$. For this task, a more useful metric is given by $\rho_{G^*_x}(g_1, g_2)$, defined in the same way as $\rho_{G_x}$ but with $G^*_x$ in place of $G_x$.  But an $\epsilon_n$ rate of convergence in $\rho_{G_x}$ also implies an $\epsilon_n$ rate of convergence in $\rho_{G^*_x}$ as long as $G^*_x$ is absolutely continuous with respect to $G_x$. Absolute continuity is unavoidable because one can hope to make accurate prediction only at points where data accumulate.

A related issue is the debate whether density regression can be essentially carried out by a nonparametric estimation of the joint density of $(X, Y)$, such as in \cite{muller96}. Our results indicate that if inference on the conditional densities of $Y$ given $X$ is of interest, then there might indeed be an advantage in pursuing the density regression formulation with the potential of obtaining much faster convergence rates through a suitable projection of $X$. From a practitioner's point of view, this would mean sharper inference, with shorter credible bands for the same amount data than what one would obtain in the joint density estimation formulation.

%\appendix
%
%\section{Appendix: Further details}
%\begin{lemma}
%\label{supp1}
%The results in Theorem \ref{thm denreg} hold if (\ref{c1})-(\ref{c3}) hold with $f_0$ and $\epsilon_n$ as in the theorem, $\scB_n \subset \rline^{\bbU_d \times[0,1]}$ and $\tilde W$ following $\gpsvstar$ or $\gplpstar$ as appropriate.
%\end{lemma}
%
%\begin{proof}
%Take $g = (g(y|x): x \in \bbU_d, y \in \rline)$ given by (\ref{den reg def}) and define $h: \bbU_d \times \rline \to \rline$ as $h(x, y) = g(y| x)$. Then $h$ defines a probability density function on $\bbU_d \times \rline$ with respect to the product of $G_x$ and the Lebesgue measure on $\bbU_d$. Denote this dominating measure by $\nu$. We can write 
%$$h(x, y) = \frac{e^{}{}$$
%
%
%\end{proof}

\bibliographystyle{chicago}
\bibliography{/Users/suryatokdar/Desktop/Utils/TokdarReferences}

\end{document}